\def\timestamp{%
Time-stamp: <unions_of_F-spaces.tex: Thursday 13-06-2013 at 09:41:02 (cest)>}
\def\stripname Time-stamp: <#1 #2>{#2}
\edef\filedate{\expandafter\stripname\timestamp}

\documentclass[a4paper]{amsart}

\usepackage{amssymb}

\DeclareMathSymbol\restr\mathbin{AMSa}{"16}  

\newtheorem{theorem}{Theorem}[section]
\newtheorem{lemma}[theorem]{Lemma}
\newtheorem{proposition}[theorem]{Proposition}
\theoremstyle{definition}
\newtheorem{definition}[theorem]{Definition}
\theoremstyle{remark}
\newtheorem{claim}{Claim}[theorem]

\newcommand\cf{\operatorname{cf}}
\newcommand\cl{\operatorname{cl}}
\newcommand\clbeta[1]{\cl_{\beta#1}}

\newcommand\orpr[2]{\langle{#1},{#2}\rangle}

\newcommand\cont{\mathfrak{c}}
\let\con\subseteq

\newcommand\interval{\mathbb{I}}
\newcommand\naturals{\mathbb{N}}
\newcommand\CH{\ensuremath{\mathsf{CH}}}

\newcommand\betaN{\beta\omega}
\newcommand\Nstar{\omega^*}
\newcommand\pistar{\pi^*}
\newcommand\preim{^\gets}

\begin{document}

\title{Unions of $F$-spaces}

\author[K. P. Hart]{Klaas Pieter Hart}
\address[K. P. Hart]%
        {Faculty of Electrical Engineering, Mathematics and Computer Science\\
         TU~Delft\\ Postbus 5031\\ 2600~GA Delft\\ The Netherlands}
\email{k.p.hart@tudelft.nl}

\author[L. Luo]{Leon Luo}
\address[L. Luo]%
        {Faculty of Electrical Engineering, Mathematics and Computer Science\\
         TU~Delft\\ Postbus 5031\\ 2600~GA Delft\\ The Netherlands}
\email{l.luo@tudelft.nl}

\thanks{The second-named author would like to thank Mr.~Yi~Tang for 
        financial support.}

\author[J. van Mill]{Jan van Mill}
\address[J. van Mill]{Faculty of Sciences,
                      VU University Amsterdam, 
                      De Boelelaan 1081A, 
                      1081~HV Amsterdam, The Netherlands}

\address[J. van Mill]%
        {Faculty of Electrical Engineering, Mathematics and Computer Science\\
         TU~Delft\\ Postbus 5031\\ 2600~GA Delft\\ The Netherlands}

\address[J. van Mill]%
       {Department of Mathematical Sciences, 
        University of South Africa,
        P.~O. Box 392,
       0003 Unisa, South Africa}
\email{j.van.mill@vu.nl}

\date{\filedate}



\begin{abstract}
We show that every space that is the union of a `small' family 
$\mathcal{F}$ consisting of special $P$-sets, is an $F$-space. 
We also comment on the sharpness of our results.
\end{abstract}

\maketitle

\section*{Introduction}

We assume that every space is Tychonoff unless specified, 
and $\beta X$ and $X^{*}$ stand for the \v{C}ech-Stone compactification and 
the \v{C}ech-Stone remainder of $X$ respectively. 
A space is an $F$-space if disjoint cozero-subsets are contained in disjoint 
zero-subsets. 
Equivalently, $X$ is an $F$-space if every cozero-subset of $X$ 
is $C^*$-embedded in $X$. 
The study of $F$-spaces has a long history since late 1950's~\cite{GJ}. 
For basic information on $F$-spaces, 
see \cite{GJ}, \cite{vanmill:betaomega} and~\cite{walker:cechstone}.

It is proven in~\cite{fg:extension} that each union of $\omega_1$~many 
cozero-subsets of an $F$-space is again an $F$-space. 
Hence under the Continuum Hypothesis (abbreviated \CH) each open subspace of 
an $F$-space of weight~$\cont$ is again an $F$-space. 
In~\cite{Dow83} an example was constructed of a compact $F$-space with 
weight $\omega_2\cdot\cont$ that has an open subspace that is not an $F$-space.
Hence \CH~is equivalent to the statement that each open subspace of an 
$F$-space with weight~$\cont$ is again an $F$-space.
See~\cite{Dow83}, \cite{Dow83a} and~\cite{DowForster82} for more related 
results.

These results have motivated us to study the question 
``when is the union of $F$-subspaces again an $F$-space'' 
more closely. 
In this note, it is shown that if a space can be covered by a family
 of $\omega_1$~many special $P$-sets, then it is an $F$-space. 
We shall also use the examples in~\cite{Dow83,Dow83a} to discuss the 
sharpness of our result.

\section{Preliminaries}

A closed subset $A$ of a space $X$ is called a $P$-set if the intersection of 
any countable family of neighborhoods of $A$ is again a neighborhood of~$A$. 
If $A$~is a singleton subset of~$X$, then the point in it
is usually referred as a $P$-point.

\begin{definition}
A space is a $P$-space if every point is a $P$-point.
\end{definition}

\begin{definition}
A closed subset $A$ of a space $X$ is called \emph{nicely placed} 
in~$X$ if for every open neighborhood~$U$ of~$A$ there is a 
cozero-subset~$V$ of~$X$ such that $A\con V \con U$.
\end{definition}

\begin{definition}
A subset $A$ of a space $X$ is said to be $C^{*}$-embedded in~$X$ if for 
each continuous function $f: A\to\interval$, there is a continuous 
extension $\bar{f}: X\to\interval$ of~$f$.
\end{definition}

\begin{proposition}[{\cite[1.61]{walker:cechstone}}]
A $C^{*}$-embedded subspace of an $F$-space is an $F$-space.
\end{proposition}

If $X$ is a set, and $\kappa$ is a cardinal number, 
then $[X]^\kappa$ denotes $\{A\con X : |A|=\kappa\}$.

\section{Unions of $F$-spaces}

In this section we present our main result on unions of $F$-subspaces. 
In the next sections we will comment on its sharpness.

\begin{theorem}\label{eerstestelling}
Let $X$ be a space with a cover $\mathcal{F}$ that consists of not more
than $\omega_1$ many $P$-subsets, each of which is a nicely placed 
$C^*$-embedded $F$-subspace of~$X$.
Then $X$~is an $F$-space.
\end{theorem}

\begin{proof}
Let $U$ be a cozero-subset of $X$, and let $f: U\to \interval$ be continuous. 
Enumerate $\mathcal{F}\cup \{\emptyset\}$ as 
$\{F_\alpha : \alpha < \omega_1\}$ where $F_0=\emptyset$. 
We shall construct, by transfinite recursion,
for each $\alpha <\omega_1$, a cozero subset~$V_\alpha$ 
of~$X$ and a continuous function $f_\alpha : V_\alpha \to \interval$
such that
\begin{enumerate}
\item $V_0 = U$ and $f_0 = f$;
\item $F_\alpha \con V_\alpha$;
\item if $\beta < \alpha$ then $V_\beta \con V_\alpha$ and 
      $f_\alpha \restr  V_\beta = f_\beta$.
\end{enumerate}
Suppose that we have constructed $V_\beta$ and $f_\beta$ for all 
$\beta < \alpha$ where $\alpha < \omega_1$. 
Put $V= \bigcup_{\beta < \alpha} V_\beta$ and 
$g = \bigcup_{\beta < \alpha} f_\beta$. 
Clearly, $V$~is a cozero-subset of~$X$ and $g$~is continuous on~$V$. 
Let $h = g\restr F_\alpha$. 
Since $V\cap F_\alpha$ is a cozero-subset of~$F_\alpha$ and 
$F_\alpha$~is an $F$-space, we can extend~$h$ to a continuous 
function $\xi: F_\alpha\to\interval$. 
Moreover, since $F_\alpha$~is $C^*$-embedded in~$X$, we can extend~$\xi$ 
to a continuous function $\eta: X\to \interval$.

We claim that there is a cozero-subset~$W$ of~$X$ such that 
$F_\alpha\con W$ and 
$$
g\restr (W\cap V) = \eta\restr (W\cap V)
$$ 
Indeed, we write $V$ as $\bigcup_{n<\omega} A_n$, 
where each~$A_n$ is closed in~$X$. 
For all $n<\omega$ and $k \ge 1$, let
$$
    A^k_n = \{x\in A_n : |g(x) - \eta(x)| \ge 2^{-k}\}.
$$
Clearly, $A^k_n$ is closed in $X$ and disjoint from~$F_\alpha$
since $g\restr  (F_\alpha \cap V) = \eta\restr  (F_\alpha \cap V)$. 
As $F_\alpha$~is a nicely placed $P$-subset of~$X$ there is a 
cozero subset~$W$ of~$X$ such that
$$
F_\alpha \con W \con X\setminus \bigcup_{n<\omega} \bigcup_{k\ge 1} A^k_n.
$$
It is clear that $W$ is as required.
Now put $V_\alpha = V\cup W$ and $f_\alpha = g \cup (\eta\restr W)$. 

At the end of the recursion $\bar f = \bigcup_{\alpha < \omega_1} f_\alpha$;
this is the desired continuous extension of~$f$.
\end{proof}

\section{The first example}

We shall describe an example of a locally compact space that is not an 
$F$-space yet it admits a clopen cover of size~$\omega_2$ 
consisting of compact zero-dimensional $F$-spaces. 
This shows that Theorem~\ref{eerstestelling} is false for unions of 
families of size~$\omega_2$.
Our example is a modification of the example in ~\cite{Dow83}.

\smallskip
Our starting point is the compact space~$G$ obtained from the topological sum
of $\Nstar\times(\omega_1+1)$ and $\betaN$ by identifying
the points $\orpr u{\omega_1}$ and $u$, for every point~$u$ of~$\Nstar$.

Observe that after this identification $\omega$ is an open $F_\sigma$-subset 
of~$G$ and that $\betaN$~is a $P$-set of character~$\omega_1$ in~$G$. 
Moreover, the weight of~$G$ is equal to~$\cont$ and $G$~is zero-dimensional.

Our next step is to put $Y = \omega\times G$. 
Let $\pi: Y\to G$ denote the projection map and let $\pistar : Y^*\to G$
be the restriction of the Stone extension of~$\pi$.
As $\pistar$~is closed the preimage $(\pistar)\preim[\betaN]$ is not open 
since $\betaN$~is not open in~$G$. 

The space $Y^*$~is a compact zero-dimensional $F$-space of weight~$\cont$ 
and $(\pistar)\preim[\betaN]$~is a $P$-set of character~$\omega_1$. 
The problem is that $(\pistar)\preim[\omega]$~is not dense 
in~$(\pistar)\preim[\betaN]$. 
To remedy this let $f$ be the restriction
 of~$\pistar$ to~$(\pistar)\preim[\betaN]$. 
Now $f$ maps the closed $P$-set $(\pistar)\preim[\betaN]$ onto 
the compact $F$-space~$\betaN$.
Hence~\cite[Lemma~1.4.1]{vanmill:betaomega} applies to show that the 
adjunction space
$
    \Omega=Y^* \cup_f \betaN
$
is a compact $F$-space of weight~$\cont$. 
It is also easily seen to be zero-dimensional.
Thus we have replaced $(\pistar)\preim[\betaN]$ in~$Y^*$ by 
(a copy of)~$\betaN$;
in this way we get an open $F_\sigma$-subset~$C$ in~$\Omega$ whose
closure is a $P$-set of character~$\omega_1$: let $C=\omega$.
 
We can give an explicit increasing sequence 
$\langle V_\alpha:\alpha\in\omega_1\rangle$
of clopen sets in~$\Omega$ such that $\omega\setminus\cl_\Omega C$
is equal to $\bigcup_{\alpha\in\omega_1}V_\alpha$.
Indeed, in $G$ we have the clopen initial segments 
of $\Nstar\times(\omega_1+1)$: put $G_\alpha=\Nstar\times(\alpha+1)$
for each~$\alpha$.  
These are transported into $Y^*$, and hence into~$\Omega$, by taking
preimages: let $V_\alpha=(\pistar)\preim[G_\alpha]$ for all~$\alpha$.

Now we perform the same construction as in~\cite{Dow83} with $\Nstar$ 
replaced by~$\Omega$. 
Let $X$ be $\omega_2+1$ endowed with $G_\delta$-topology.
We observe that $X\times \Omega$ is an $F$-space by~\cite{Negrepontis69}, 
and that its weight is equal to $\omega_2\cdot\cont$.
This implies that $K = \beta (X\times \Omega)$ is an $F$-space as well
and its weight is equal to~$(\omega_2\cdot\cont)^\omega=\omega_2\cdot\cont$.

Next let $L= \{\alpha\in\omega_2+1 : \cf\alpha\ge\omega_1\}$.
We let $T$ be the closure in~$K$ of $L\times C$; note that
$T=\cl_K(L\times\cl_\Omega C)$ also.
The complement~$U$ of $T$ in $K$ is our example.

That $U$ is not an $F$-space is proven in exactly the same way 
as in~\cite{Dow83}.

To finish we show that $U$ is the union of $\omega_2$ many clopen subset of~$K$.
Each of these is trivially a nicely placed and $C$-embedded $P$-set, and
an $F$-space because $K$~is.

The first $\omega_1$ many clopen sets are the closures 
$\cl_K(X\times V_\alpha)$, for $\alpha\in\omega_1$; these cover the points
of~$U$ that do not belong to $\cl_K(X\times\cl_\Omega C)$, 
as we shall see presently.

The other $\omega_2$ many clopen sets will appear in the course of the 
following argument.
Let $u\in U$ and let $W$ be a clopen neighbourhood of~$u$ in~$K$ that is 
disjoint from~$T$.
We let $A =\bigl\{\alpha\in X\setminus L : 
 (\exists m\in\cl_\Omega C)(\orpr{\alpha}{m}\in W)\bigr\}$;
note that, because $W$~is clopen, it is even the case that 
$W\cap(\{\alpha\}\times C)\neq\emptyset$ whenever $\alpha\in A$.

\begin{claim}
$A$ is countable.
\end{claim}
\begin{proof}
If $A$ is uncountable then, as a set or ordinals, it has an initial
segment of order type~$\omega_1$; we simply assume that the order type 
of $A$ itself is~$\omega_1$.
Let $\beta = \sup A$. 
Then $\beta \in L$. 
Moreover, for every $\alpha\in A$, pick, by the above remark,
an element $m_\alpha\in C$ such that $\orpr{\alpha}{m_\alpha}\in W$. 
Since $C$ is countable there is an $m\in C$ such that
$M=\{\alpha:m_\alpha=m\}$ has cardinality~$\omega_1$.
This then implies that
$\orpr{\beta}{m}\in W\cap (L\times C)$, a contradiction.
\end{proof}

\begin{claim}
There exists $\alpha < \omega_1$ such that
$$
    W\cap (X\times \Omega) \con (X\times V_\alpha) \cup (A\times \Omega).
$$
\end{claim}

\begin{proof}
To begin we observe that for every $\gamma\in X\setminus A$ there
is an~$\alpha$ such that
$$
W\cap(\{\gamma\}\times\Omega)\subseteq \{\gamma\}\times V_\alpha
$$
This follows because $W\cap(\{\gamma\}\times\Omega)$ is compact
and disjoint from $\{\gamma\}\times\cl_\Omega C$.

We claim that for each $\alpha$ the set
$O_{\alpha}=\{\gamma\notin A:
  (\{\gamma\}\times\Omega)\cap W\subseteq\{\gamma\}\times V_{\alpha}\}$ 
is open in~$X$.
Indeed, 
$X\setminus O_\alpha=A\cup\pi_X[W\cap(X\times(\Omega\setminus V_\alpha))]$, and 
this set closed because $A$ is closed and because the projection
$\pi_X:X\times(\Omega\setminus V_\alpha)\to X$ is closed (by compactness
of $\Omega\setminus V_\alpha$.

By repeated application of the pressing-down lemma one readily proves that
$X\setminus A$ is Lindel\"of, so that there is $\beta\in\omega_1$ such that
$X\setminus A\subseteq\bigcup_{\alpha<\beta}O_\alpha$.

But this then implies that 
$W\cap (X\times \Omega) \con (X\times V_\beta) \cup (A\times \Omega)$.
\end{proof}

Since $(X\times V_\beta) \cup (A\times \Omega)$ is clopen in $X\times\Omega$
we see that $W\subseteq \cl_K(X\times V_\beta)\cup\cl_K(A\times \Omega)$.

From this we extract our second family of clopen sets: all sets
of the form $\cl_K(A\times \Omega)$ for countable $A\subseteq X\setminus L$.

We finish by observing that $[X\setminus L]^\omega$ has a cofinal subfamily
$\mathcal{A}$ of cardinality~$\omega_2$:
for each $\alpha\in\omega_2$ the set $[\alpha\setminus L]^\omega$ has a cofinal 
subfamily~$\mathcal{A}_\alpha$ of cardinality~$\omega_1$, obtained via an 
injection from~$\alpha$ into~$\omega_1$.
Then $\mathcal{A}=\bigcup_{\alpha<\omega_2}\mathcal{A}_\alpha$ is as required.

Hence the clopen families $\{\cl_K(X\times V_\alpha):\alpha\in\omega_1\}$
and $\{\cl_K(A\times \Omega): A\in \mathcal{A}\}$ is the required cover of~$U$.

\section{The second example}

We shall describe an example of a space that admits a cover of size~$\omega_1$ 
consisting of $C^*$-embedded $F$-subspaces that are $P$-sets yet it is not
an $F$-space. 
This shows that Theorem~\ref{eerstestelling} is false for unions of $P$-sets 
that are not nicely placed.
The space is Example~1.9 from~\cite{Dow83a}.

Let $X = \omega_1\cup \{p\}$, 
where neighborhoods of~$p$ are cocountable and $\omega_1$~is discrete. 
Let $S = \omega_1\times \Nstar$, where again $\omega_1$~has the discrete 
topology. 
Let $C\con \Nstar$ be a cozero subset whose closure is not a zero-set. 
For $\alpha\in \omega_1$, let $C_\alpha = \{\alpha\}\times C$, and put
$$
K = \bigcap_{\alpha\in\omega_1} 
  \clbeta{S}\left(\bigcup_{\gamma>\alpha} C_\gamma\right).
$$
Then $Y=\beta S\setminus K$ is a locally compact $F$-space, 
and $X\times Y$ is not an $F$-space~\cite{Dow83a}.

The crucial property of $Y$ is the following: 
if for each $\alpha$ one takes a zero subset~$Z_\alpha$ of 
$\{\alpha\}\times \Nstar$ that contains $C_\alpha$ then
$$
Y\cap\bigcap_{\alpha\in\omega_1}
 \clbeta{S}\left(\bigcup_{\gamma>\alpha} Z_\gamma\right) \neq\emptyset. 
\leqno{(\dagger)}
$$

\begin{lemma}
Let $x$ be a $P$-point in a space~$D$ and 
let $E$ be a locally compact space. 
Then $\{x\}\times E$ is a $P$-set in $D\times E$.
\end{lemma}

\begin{proof}
Let $F$ be an $F_\sigma$-subset of $D\times E$ which is disjoint from 
$\{x\}\times E$, we show that $\cl F$ is also disjoint from $\{x\}\times E$.
 
To this end let $y\in E$ and let $C$ be a compact neighborhood of~$y$ in~$E$. 
The projection map $\pi_D: D\times C\to D$ is closed, hence 
$H=\pi_D[F \cap (D\times C)]$ is an $F_\sigma$-subset of~$D$ that does not 
contain~$x$. 
Hence $U = D\setminus\cl H$ is a neighborhood of~$x$ since $x$~is a $P$-point 
in~$D$. 
So the product of~$U$ and the interior of~$C$ is a neighborhood of 
$\orpr{x}{y}$ that is disjoint from~$F$, so that $\orpr{x}{y}\not\in\cl F$. 
\end{proof}

From this Lemma we conclude that the collection
$$
    \bigl\{\{x\}\times Y: x\in X\bigr\}
$$
consists of $P$-subsets of $X\times Y$ that are themselves $F$-spaces
and clearly $C^*$-embedded.
Since $X\times Y$ is not an $F$-space, at least one of them cannot be nicely 
placed by Theorem~\ref{eerstestelling}. 
Since $\{q\}\times Y$ is clopen in $X\times Y$ for every $q\in X\setminus \{p\}$
the only candidate for such $P$-set is $E=\{p\}\times Y$. 
It is instructive to provide a direct argument that $E$~is not nicely placed 
in~$X\times Y$.

To this end put
$$
 A = \bigcup_{\alpha\in \omega_1} \{\alpha\}\times C_\alpha.
$$
It was shown in the proof of Theorem~1.7 in~\cite{Dow83a} that $A$ is a 
cozero subsets of~$Y$. 
Since $A$~is disjoint from the $P$-set $E$ there is a neighbourhood~$O$ 
of~$E$ that is disjoint from~$A$.
If $E$ were nicely placed in $X\times Y$ then there would be a cozero-set~$V$
in~$X\times Y$ such that $E\con V \con O$. 
Hence $Z= (X\times Y)\setminus V$ is a zero-set in $X\times Y$ that 
contains~$A$ but misses~$E$. 
For every $\alpha < \omega_1$, put 
$Z_\alpha = Z\cap (\{\alpha\} \times \Nstar)$, this is a zero-set 
in $\{\alpha\} \times \Nstar$ that contains~$C_\alpha$. 
By ~$(\dagger)$ the intersection
$$
Y\cap\bigcap_{\alpha\in\omega_1}
 \clbeta{S}\left(\bigcup_{\gamma>\alpha} Z_\gamma\right) 
$$
is nonempty.
This intersection is a subset of~$Z\cap E$ which was assumed to be empty.

\section{The third example}

The only question left is whether the hypothesis of being $C^*$-embedded is 
essential for Theorem~\ref{eerstestelling}. 
Unfortunately, we are unable to answer this question. 
A simpler question is: 
Is it true that every $P$-subset which is nicely placed in an $F$-space 
is $C^*$-embedded in that space? 
If the answer is positive, the condition on $C^*$-embeddedness in 
Theorem~\ref{eerstestelling} would be superfluous. 
We can show that the assumption $2^{\omega_1} = \omega_2$ implies the answer 
is negative.
 
The equality $2^{\omega_1}=\omega_2$ implies that there is a maximal almost 
disjoint family on~$\omega_1$ of cardinality~$2^{\omega_1}$,
that is, a collection~$\mathcal{A}$ of subsets of~$\omega_1$ with the 
following properties:
\begin{enumerate}
\item $\mathcal{A}\con [\omega_1]^{\omega_1}$,
\item if $A,B\in \mathcal{A}$ are distinct, then $|A\cap B| \le \omega$,
\item $\mathcal{A}$ is maximal with respct to the properties (1) and (2),
\item $|\mathcal{A}| = 2^{\omega_1}$.
\end{enumerate}
Let $X$ be $\omega_1\cup \mathcal{A}$ and topologize $X$ in the standard way 
as follows: the points of $\omega_1$ are isolated and a neighborhood 
of~$A\in \mathcal{A}$ contains $\{A\}$ and all but countably many elements 
from~$A$. 
Then $X$ is a $P$-space, and by Jones' Lemma, 
the set~$\mathcal{A}$ is not $C^*$-embedded in $X$. 
However, by maximality of $\mathcal{A}$, every neighborhood of~$\mathcal{A}$ 
has a countable complement and is therefore clopen. 
So, every neighborhood of $\mathcal{A}$ is clopen, 
and therefore $\mathcal{A}$ is nicely placed in the $P$-space $X$ for 
trivial reasons.


\def\cprime{$'$}
\makeatletter \renewcommand{\@biblabel}[1]{\hfill[#1]}\makeatother

\end{document}